\documentclass[english]{article}
\usepackage[T1]{fontenc}
\usepackage[latin9]{inputenc}
\usepackage{amsmath}
\usepackage{amssymb}
\usepackage{babel}
\usepackage{amsthm}
\usepackage{graphicx}
\usepackage{tikz}
\usepackage[all]{xy}
\usepackage{bm}
\usepackage{comment}
\usetikzlibrary{calc,patterns,angles,quotes}
\usepackage[toc,page]{appendix}
\usepackage{hyperref}

\newcommand{\bl}{\textcolor{blue}}

\newtheorem{theorem}{Theorem}
\newtheorem{prop}[theorem]{Proposition}
\newtheorem{lemma}[theorem]{Lemma}

\newtheorem{rem}[theorem]{Remark}
\newtheorem{cor}[theorem]{Corollary}
\newtheorem{defi}[theorem]{Definition}

\newcommand{\R}{\mathbb{R}}

\newcommand{\Z}{\mathbb{Z}}
\newcommand{\C}{\mathbb{C}}

\title{Kustaanheimo-Stiefel Transformation, Birkhoff-Waldvogel Transformation and Integrable Mechanical Billiards}
\author{Airi Takeuchi, Lei Zhao}

\begin{document}

\maketitle

\begin{abstract}

The three-dimensional Kepler problem is related to the four-dimensional isotropic harmonic oscillators by the Kustaanheimo-Stiefel Transformations. In the first part of this paper, we study how certain integrable mechanical billiards are related by this transformation. This in part illustrates the rotation-invariance of integrable reflection walls in the three-dimensional Kepler billiards found till so far. The second part of this paper deals with Birkhoff-Waldvogel's Transformation of the three-dimensional problem wiht two Kepler centers. In particular, we establish an analogous theory of Levi-Civita planes for Birkhoff-Waldvogel's Transformation and showed the integrability of certain three-dimensional 2-center billiards via a different approach.

\end{abstract}

\section{Introduction}
{The Kustaanheimo-Stiefel transformation was introduced in \cite{Kustaanheimo, KS} using spinors as a way to regularize the 3-dimensional Kepler problem. This transformation also admits formulation with quaternions \cite{Waldvogel Right Way, Saha, Zhao KS}, and can be regarded as an unfolding of the Levi-Civita regularization  \cite{Levi-Civita 1904, Levi-Civita} of the planar Kepler problem through the theory of Levi-Civita planes \cite{Stiefel-Scheifele, Zhao KS}.

The use of conformal Levi-Civita transformation \cite{MacLaurin, Goursat, Levi-Civita} to study planar integrable mechanical billiards defined with the Hooke and Kepler problems has been first pointed out in \cite{Panov} and extended in \cite{Takeuchi-Zhao Conformal}. 

In the first part of this note, we discuss some consequences of the K.-S. transformation on integrable 4-dimensional Hooke and integrable 3-dimensional Kepler billiards.
It is widely known that for the 4-dimensional Hooke problem, a centered quadric reflection wall gives an integrable billiard system \cite{Jacobi Vorlesung}, \cite{Fedorov}. We shall show that when this reflection wall is invariant under the $S^{1}$-symmetry of the K.S. transformation, then its image under the Hopf mapping is one of the five special type of quadrics, with the Kepler center as a focus. This is in consistence with the results from \cite{Takeuchi-Zhao Projective Space} and provides a partial explanation of why only these quadrics appearsin 3-dimensional integrable Kepler billiards. We can think that the restriction on the type of quadrics is forced by the $S^{1}$-invariance of the centered quadric reflection wall lying on the 4-dimensional Hooke side. This generalizes the studies in the planar case \cite{Takeuchi-Zhao Conformal}, \cite{Panov} to the spatial case of Kepler billiards based on the Levi-Civita transformation. Moreover, though we shall not discuss this aspect in this article, the method is not limited to integrable mechanical billiard systems.


The Kustaanheimo-Stiefel transformation has been extended to a transformation which simultaneously regularize both double collisions in a spatial two-center problem first announced in the 1-page note of Stiefel-Waldvogel \cite{Stiefel-Waldvogel}, which generalized the transformation of Birkhoff used in the planar case. The thesis of Waldvogel \cite{Waldvogel Diss} provided a much more extensive, geometrical study of this transformation. In particular the relation between this transformation and the Kustaanheimo-Stiefel transformation has been clarified. Waldvogel later illustrated this theory again in \cite{Waldvogel Right Way} with the use of Quaternions. In this article, we provide a quaternionic formulation of this Birkhoff-Waldvogel transformation in the spatial case, largely inspired by the studies of Waldvogel as well as combining the symplectic viewpoint of \cite{Zhao KS}. We investigate in part an analogous theory of Levi-Civita planes in this setting, consisting of planes and spheres in the space of quaternions $\mathbb{H} \cong \R^{4}$ and a reduction of this transformation to a dense open subset of $\mathbb{IH} \cong \R^{3}$, which already regularizes the double collisions without increasing the dimension of the space. With this we link integrable billiards on both sides, which illustrates some results in \cite{Takeuchi-Zhao Projective Space} with a different method.

We organize this article as follows: In Section \ref{Sec: KS}, we recall the theory of Kustaanheimo-Stiefel regularization, which largely follows \cite{Zhao KS}. Then we apply this transformation to link integrable mechanical billiards in Section \ref{Sec: integrable billiards}. The theory of Birkhoff-Waldvogel transformation and the corresponding link on integrable mechanical billiards are discussed in Section \ref{sec: 2CP}.

\section{The Kustaanheimo-Stiefel Transformation} \label{Sec: KS}


{In this section, we discuss the Kustaanheimo-Stiefel transformation. We {follow} the quaternionic formulation {of \cite{Zhao KS}.}}

{A quaternion is represented as}
\[
z= z_0 + z_1 i+z_2 j+z_3 k, \quad z_0,z_1, z_2,z_3 \in \R
\]
{in which 
$$i^{2}=j^{2}=k^{2}=-1, ij=-ji=k, jk=-kj=i, ki=-ik=j.$$
Additions and multiplications of quaternions are then naturally defined. With these operations, the quaternions form a non-commutative normed division algebra which we denote by $\mathbb{H}$. } For a quaternion $z= z_0 + z_1 i+z_2 j+z_3 k$, its real part is given by 
$$Re(z)= z_0$$
and its imaginary part is given by 
$$Im(z)= z_1 i+z_2 j+z_3 k.$$ 
{Further, the conjugation of $z$ is defined as 
$$\bar{z}=z_0 - z_1 i-z_2 j-z_3 k.$$ 
The norm of $z$ is defined as $|z|:=\sqrt{z \cdot\bar{z}}$.}

 We denote the set of purely {imaginary quaternions} by 
 $$\mathbb{IH}=\{z \in \mathbb{H} \mid Re(z) = 0 \}.$$
 We identify $\mathbb{H}$ with $\R^4$ and denote by $\mathcal{S}^3$ the unit sphere
\[
\{ z \in \mathbb{H} \mid  |z|^2= z_0^2 + z_1^2+z_2^2+z_3^2=1 \} \subset \mathbb{H}.
\]
Also, we identify $\mathbb{IH}$ with $\R^3$. The unit sphere $\mathcal{S}^{2}$ there in is 
\[
 \{z \in \mathbb{IH} \mid  |z|^2=  z_1^2+z_2^2+z_3^2=1 \} \subset \mathbb{IH}.
\]

{To introduce the Kustaanheimo-Stiefel transformation, we first recall the Levi-Civita transformation \cite{Levi-Civita}
\[
T^*(\C \setminus \{0  \}) \to T^*(\C \setminus \{0  \}),  (z,w) \mapsto \left(q=z \cdot z, p=\dfrac{z}{2 |z|^{2}} \cdot w \right).
\]
It is well-known that this transformation is {canonical}, and {transforms the planar Kepler problem into the planar Hooke problem after making a proper time reparametrization  on a energy level.} 
{To see this, we start with the shifted Hamiltonian of the Kepler problem and consider its zero-energy level:
\[
\frac{|p^2|}{2} + \frac{m}{|q|} + f = 0.
\]
The Levi-Civita transformation pulls this system back to 
\[
\frac{|w^2|}{8|z|^2} + \frac{m}{|z|^2} + f = 0.
\]
We may now multiply this transformed Hamiltonian by $|z|^2$, which only reparametrizes the flow on this energy-level.  We obtain
\[
\frac{|w^2|}{8} + m + f|z|^2 = 0.
\]
which is the restriction of the Hamiltonian of the planar Hooke problem $\frac{|w^2|}{8} + f|z|^2$ on its $-m$-energy hypersurface.}

The whole construction is based on the complex square mapping 
$$\C \mapsto \C, \qquad z \mapsto z^{2},$$
which is a $2$-to-$1$ conformal mapping.}

A generalization of the complex square mapping {with quaternions} is  the {following \emph{Hopf mapping} }
\[
\mathbb{H} \to  \mathbb{IH}, \qquad z \mapsto \bar{z}i z.
\]
{{Note that this mapping is well-defined, since 
$$Re( \bar{z}i z) = 0, \forall z \in \mathbb{H}.$$  }

This mapping is ``$S^{1}$-to-$1$'', namely the circle 
$$\{\exp(i \theta)z\mid ,z \in \mathbb{H},\,\,  \theta \in {\R/2\pi\Z} \} \subset \mathbb{H}$$
is mapped under the Hopf mapping to the same point $\bar{z} i z \in \mathbb{IH}$. 

Moreover, this mapping is restricted to a mapping $\mathcal{S}^3 \to \mathcal{S}^2$. This is a mapping with $S^{1}$-fibres, and induces the non-trivial Hopf fibration 
$$S^{1} \hookrightarrow \mathcal{S}^3 \to \mathcal{S}^2.$$}

{{Associated to the Hopf mapping, the Kustaanheimo-Stiefel mapping is defined as}
\[
T^*(\mathbb{H} \setminus \{ 0\}) \to \mathbb{IH} \times \mathbb{H},  (z,w) \mapsto \left(Q=  \bar{z}i \cdot z,  P= \frac{ \bar{z}i }{2 |z|^2} \cdot w \right).
\]
The fibers of the mapping {are the circle} orbits of the $S^{1}$-Hamiltonian action 
$${\theta \cdot} (z,w) \mapsto (\exp(i \theta) z,  \exp(i \theta) w).$$}
{on the cotangent bundle $T^{*} \mathbb{H}$.} {The {bilinear} function }
\[
BL(z,w):= Re(\bar{z} i w)
\]
 {is the associated moment map}.
 
{We define 
$$\Sigma:=\{(z, w) |BL(z,w) = 0 \} \subset T^*\mathbb{H} \cong \mathbb{H} \times \mathbb{H},$$ 
and 
$$\Sigma^1 = \Sigma \setminus \{ z=0 \};$$ 
Both are invariant under this $S^{1}$-Hamiltonian action.}

{We define the restricted K.S mapping as
$$KS:=K.S.|_{\Sigma_{1}}: \Sigma^1 \to T^*(\mathbb{IH}\backslash \{ 0 \}).$$}}

{For the following lemmaf rom \cite{Zhao KS}, we present an alternative, simpler  proof}.
\begin{lemma}
	\label{lem: KS_sympl}
	{For the restricted Kustaanheimo-Stiefel mapping $KS: \Sigma^1 \to T^*(\mathbb{IH}{\backslash \{ 0 \}})$ we have}
	\[
	KS^*(Re(d \bar{P}\wedge dQ)) = Re(d \bar{w} \wedge dz)|_{\Sigma^1}. 	\]

\end{lemma}
\begin{proof}
	{We shall show
	\begin{equation}\label{eq: one-form 0}
	{KS^*Re(\bar{P} d Q) = Re(\bar{w} d z)|_{\Sigma^1}.}
	\end{equation}
	which then implies the assertion of this lemma by taking differentials on both sides.
	
	To see \eqref{eq: one-form 0}, we compute 
	\begin{equation}
	\label{eq: one-form_1}
	\begin{split}
	\bar{P} dQ = - \frac{\bar{w}i \bar{z}^{-1}}{2}((d\bar{z})i z + \bar{z} i dz) \\
	= {(- \bar{w} i \bar{z}^{-1} (d \bar{z})iz+ \bar{w}dz)/2.}
	\end{split}
	\end{equation}
	{The condition 
	$$BL(z,w)=Re(\bar{z}iw)=0$$ 
	is equivalent to 
	$$Re(z^{-1}i w)=0.$$
	Consequently, we also have 
	$$Re(\bar{w}i \bar{z}^{-1}) = 0.$$
	
	This implies
	$$
	Re(\bar{w} i \bar{z}^{-1} (d \bar{z})iz)  = Re(\bar{w} i \bar{z}^{-1} \cdot Im((d \bar{z})iz)).
	$$
	Since 
	$$Im((d \bar{z})iz)= -Im( \bar{z} (-i)dz),$$
	we have
	\begin{equation}
	\label{eq: one-form_2}
	Re(\bar{w} i \bar{z}^{-1} (d \bar{z})iz) = -Re(\bar{w} i \bar{z}^{-1} \cdot Im( \bar{z} (-i)dz)) =-Re(\bar{w} i \bar{z}^{-1}  \bar{z} (-i)dz)= -Re(\bar{w} dz),
	\end{equation}
	where in the second equation, we have used 
	$$Re(\bar{w}i \bar{z}^{-1}) = 0.$$
	The assertion \eqref{eq: one-form 0} is thus obtained by combining the equations \eqref{eq: one-form_1} and \eqref{eq: one-form_2}.}
	}
\end{proof}

{{On $\Sigma^{1}$,  the orbits of the $S^{1}$-action mentioned above lie in the direction of the 1-dimensional kernel distribution of the 2-form $ Re(d\bar{w}\wedge dz )$.}  
By the theory of symplectic reduction, the {2-}form $ Re(d\bar{w}\wedge dz )$ of $\Sigma^1$ gives rise {to the} reduced symplectic form $\omega_1$ on the quotient space {$V^1$ of $\Sigma^{1}$ by the $S^{1}$-action}.
Thus, {the Kustaanheimo-Stiefel mapping induces a symplectomorphism 
$$KS_{red}: (V_1, \omega_1) \to (T^*(\mathbb{IH}\backslash \{ 0 \}),Re(d \bar{P}\wedge dQ) ).$$
We have 
$$KS=KS_{red} \circ \phi$$
in which $\phi: \Sigma^{1} \to V_{1}$ is the quotient map.} 

}


\begin{prop} \label{prop: 01} 
	{Any zero-energy orbit of the 4-dimensional Hooke problem with the shifted Hamiltonian
	$$\frac{\|w\|^2}{8} +  f \| z\|^2 +m$$
	in $\Sigma^{1}$ is sent via $KS$ to a zero-energy orbit of the 3-dimensional Kepler Problem in $T^{*} (\mathbb{IH}\backslash \{ 0 \})$ with Hamiltonian
	$$\frac{\|P\|^2}{2} + \frac{m}{\| Q\|} + f .$$ 
	after a proper time reparametrization.}

	
\end{prop}
\begin{proof}
          {We first observe that the function $BL$ is a first integral of the system 
          $$\left(T^{*} \mathbb{H}, Re(d\bar{w}\wedge dz),  H=\frac{\|w\|^2}{8} +  f \| z\|^2 +m\right).$$ {{This}  follows either from a direct verification, or alternatively from the invariance of $H$ under the above mentioned {(Hamiltonian)} $S^1$-action.} 
           Consequently, the set $\Sigma^{1}$ is invariant under its flow. 
           
           We consider the restriction of this system on $\Sigma^{1}$. Any orbit of this restricted system descends to an orbit in the quotient system in $(V_{1}, \omega_{1}, {H_1})$ so that 
          $$\phi^{*}{H_1}=H,$$
           which is consequently sent to an orbit via $KS_{red}$ in the system 
           $$(T^{*} \mathbb{IH}, Re(d\bar{p}\wedge d q), {K})$$
            such that 
          $$KS_{red}^{*}\, {K}={H_1}.$$
          Applying $\phi^{*}$ to both sides of this identity, we get
          $$H=\phi^{*} KS_{red}^{*} \, {K}=KS^{*}\, {K}.$$ 
          From this we deduce 
          $${K}=\frac{\|P\|^2 \|Q\|}{2} + m + f \|Q\|. $$
          Now we restrict the system to {$\{{K}=0\}=\{H=0\}$. We} observe that the restricted flow can now be time reparametrized {(with factor $\|Q\|^{-1}$)} into the { restricted flow on the zero-energy hypersurface} of the 3-dimensional Kepler Hamiltonian
          $${\frac{\|P\|^2 }{2} + \frac{m}{\|Q\|} + f }.$$
        }
    \end{proof}
        
       {The link between Kustaanheimo-Stiefel transformation and the Levi-Civita transformation is given by the Levi-Civita planes. These are planes in $\mathbb{H}$ generated by two unit quaternions $v_{1}, v_{2}$ such that $v_{1} \neq \pm v_{2}$ and satisfy $$BL(v_{1}, v_{2})=0.$$
        The key property of such a plane is that its image is a plane in $\mathbb{IH}$ and in relevant basis the restriction of the Hopf mapping is equivalent to the complex square mapping 
$$\C \to \C, z \mapsto z^{2}.$$
Therefore $K.S.$ is restricted to $L. C. $ on the tangent bundle of such a plane. We proceed with the details. 

\begin{defi}
	A Levi-Civita plane is a plane in $\mathbb{H}$ spanned by two linearly independent unit quarternions $v_1, v_2 \in \mathbb{H}$ satisfying $BL(v_1, v_2)=0$.
\end{defi}

\begin{prop}
	The Hopf mapping
	\[
	\mathbb{H} \to \mathbb{IH}, \quad z \mapsto \bar{z}iz
	\]
	sends a Levi-Civita plane to a plane passing through the origin in $\mathbb{IH}$. On the {other} hand, any plane in $\mathbb{IH}$ passing through the origin is the image of a $\mathcal{S}^1$-family of Levi-Civita planes.
\end{prop}
\begin{proof}
	Let $V$ be a Levi-Civita plane spanned by two unit, {orthogonal quaternions}  $v_1$ and $v_2$ in $\mathbb{H}$: {This means that we have
	$$|v_{1}|=|v_{2}|=1, \, BL(v_1, v_2)=0\, \hbox{and} \, \langle v_1, v_2 \rangle=0.$$}
	 Then, we have \[
	\bar{v}_1 i v_1 = -\bar{v}_2 i v_2,
	\]
	{which follows from the computation }
	\begin{equation}
	\label{eq: LC_plane_1}
	\begin{split}
	2\bar{v}_1 i v_1 + 2\bar{v}_2 i v_2 &= (\bar{v}_2v_1 + \bar{v}_1v_2)(\bar{v}_1i v_2+ \bar{v}_2i v_1) \\
	&= 2 \langle v_1, v_2 \rangle (\bar{v}_1i v_2+ \bar{v}_2i v_1)\\
	&=0,
	\end{split}
	\end{equation}
	{For the first equation in \eqref{eq: LC_plane_1} we used the following fact: The condition 
	$$BL(v_1, v_2)=0$$ 
	is equivalent to 
	$$\bar{v}_1i v_2 - \bar{v}_2 i v_1 =0.$$} {Thus}
	$$
	(\bar{v}_1 v_2 - \bar{v}_2 v_1)(\bar{v}_1i v_2 - \bar{v}_2 i v_1)= 0 
	$$
	which is equivalent to
	$$
	 \bar{v}_1 i v_1 + \bar{v}_2 i v_2 = \bar{v}_1 v_2 \bar{v}_1 i v_2 + \bar{v}_2 v_1 \bar{v}_2 i v_1. 
	$$

	Thus the {quaternion} $ v_1 +  v_2$ in $V$ is sent {via the Hopf mapping} to the {quaternion}
	\[
	\bar{v}_1i v_1 +  \bar{v}_1 i v_2 +  \bar{v}_2 i v_1 + \bar{v}_1 i v_1 = 2 \bar{v}_1 i v_2.
	\]
	{As a} vector in $\mathbb{IH}$, {it} is linearly independent of the vector $\bar{v}_1 i v_1$, which follows from $\bar{v}_{1} i \neq 0$ and the linear independency of $v_{1}$ and $v_{2}$. 
	
	{As a consequence, the image of $V$ is the plane passing through the origin, linearly spanned by $\bar{v}_1 i v_1$ and $\bar{v}_1 i v_2$.}
	
	
	 
	 {On the other hand}, for any unit {quaternion} $w \in \mathbb{IH}$, there {exists} a $S^1$-family of unit vectors $ \{e^{i \theta}v\}$ in $\mathbb{H}$ whose image under the Hopf map is $w$. Take a plane $W$ in $\mathbb{IH}$ passing through the origin spanned by two be linearly independent unit vectors $w_1$ and $w_2$. We can choose the {pre-images} of $v_1$ and $v_2$ in $\mathbb{H}$ of $w_1$ and $w_2$ to be such that $BL(v_1, v_2) = 0$.
	 {Indeed, {for} $\bar{v}_1 i v_2 = z_0 + z_1 i + z_2 j +z_4 k$, we have 
	 \[
	 Re(e^{i \theta}\bar{v}_1 i v_2 ) = z_0\cos \theta- z_1 \sin \theta,
	 \]
	 thus we can take $e^{i {\theta_{1}}} v_1$ such that $z_0\cos {\theta_{1}}- z_1 \sin {\theta_{1}}=0$ in the place of $v_1$.} 
	 Thus we get {the} family of Levi-Civita planes $\hbox{span} \{e^{i \theta}v_1, e^{i \theta }v_2\}, \theta \in \R/{2 \pi \Z}$ which are sent to $W$.
\end{proof}

\begin{prop}
	\label{prop: LC-Hopf}
	There exist an identification to $\C$ of a Levi-Civita plane $V$ {together with} its image {under} the Hopf mapping{, }such that under this identification, the {restriction of the K.S. mapping}  to $T^*V$ is given by 
	\[
	T^* \C \to T^* \C, \quad (z,w) \mapsto (z^2, {\dfrac{z}{2 |z|^{2}} \cdot w})
	\]
	which is the Levi-Civita transformation.
\end{prop}
\begin{proof} 
	Let $v_1$ and $v_2$ be orthogonal unit vectors in $V$, which allows us to identify $V$ with $\C$. We write $z= a v_1 + b v_2$ and $w = c v_1 + d v_2$. Then K.S. sends $(z,w)$ into
	\[
	((a^2- b^2)\bar{v}_1 i v_1 + 2 ab \bar{v}_1 i v_2, \frac{(ac-bd)\bar{v}_1 i v_1 + (ad +bc)\bar{v}_1 i v_2}{2(a^2+b^2)}).
	\]
	From the orthogonality of $v_1$ and $v_2$, we obtain 
	\[
	\langle \bar{v}_1 i v_1, \bar{v}_1 i v_2 \rangle = \frac{\bar{v}_1 v_2 + \bar{v}_2 v_1 }{2} = \langle v_1 , v_2 \rangle = 0.
	\]
	Hence we just need to identify $\bar{v}_1 i v_1$ and $\bar{v}_1 i v_2$ with the standard orthogonal basis of $\C$. {The conclusion follows after both $V$ and its image have been identified to $\C$.}
\end{proof}

\section{Application to integrable Hooke and Kepler billiards } \label{Sec: integrable billiards}

We extend the correspondence shown above to the corresponding billiard systems. This generalizes the planar correspondence of Hooke and Kepler billiards \cite{Panov}, \cite{Takeuchi-Zhao Conformal} to the spatial Kepler case. 


{A centered quadric in $\mathbb{H} \cong \R^{4}$ is called $S^{1}$-invariant, if it is invariant under the $S^{1}$-action 
$$S^{1} \curvearrowright \mathbb{H}, \quad \theta \cdot z \mapsto \exp(i \theta)z.$$
Equivalently, these are quadrics which are pre-images of subsets in $\mathbb{IH}$ under the Hopf mapping.}

A centered quadric in $\mathcal{H}$ is called non-singular if it does not contain the origin.

	{For an unbounded non-singular centered quadric in $\R^4$ given by 
	\[
	F(z_0,z_1,z_2,z_3) =1
	\]
	where $F$ is a quadratic homogeneous function of $z= (z_0,z_1,z_2,z_3) \in \mathbb{H}$,
	we define its dual quadric by 
	\[
	-F(z_0,z_1,z_2,z_3) =1.
	\]
	}
	{In normal form, for the quadric
	\[
	\sum_{i= 0}^{3} \sigma_i \frac{\hat{z}_i^2}{a_i^2}=1,
	\]
	where $\sigma_i \in \{1, -1\}, a_i \in \R$ and $\{\hat{z}_0\}_{i = 0}^3$ is an orthonormal basis in $\R^4$, its dual is
  	\[
  	\sum_{i= 0}^{3} - \sigma_i \frac{\hat{z}_i^2}{a_i^2}=1.
  	\]
	Indeed for a quadric homogeneous function $F(z_0, z_1,z_2,z_3)$ there exist a real symmetric $4 \times 4$ matrix $A$ and a real orthogonal matrix $Q$ such that $z^{T}Az = F$ and $Q^{T}AQ$ is diagonal, thus its normal form is given by $(Qz)^{T}A Qz =1$. Clearly, we have $z^{T}(-A)z = -F$ and $ (Qz)^{T}(-A) Qz = -(Qz)^{T}A Qz $.}

\begin{lemma}
 \label{lem: least_dist_dual}
 {For an unbounded non-singular centered quadric $\mathcal{E}$ and its dual quadric $\tilde{\mathcal{E}}$ in $\mathbb{H}$, we denote their images in $\mathbb{IH}$ by the Hopf mapping by $\mathcal{F}$ and $\tilde{\mathcal{F}}$ respectively. 
 Let $P \in \mathcal{F}$ be the point of $\mathcal{F}$ with the least distance from $O \in \mathbb{IH}$. Let $\tilde{P} \in \tilde{\mathcal{F}}$ be the point of $\tilde{\mathcal{F}}$ with the least distance from $O \in \mathbb{IH}$. Then the three points $O, P, \tilde{P}$ are collinear.}
\end{lemma}
\begin{proof}
	{{Consider a plane contains $O, P, \tilde{P}$ such that the intersection of $\mathcal{F}$ is unbounded. Then the intersection of $\tilde{\mathcal{F}}$ is unbounded as well.} If $\mathcal{E}$ is non-degenerate, then the intersections of $\mathcal{F}$ and $\tilde{\mathcal{F}}$ with this plane are either two centrally symmetric parallel lines or two branches of a focused hyperbola, since they are the image of a pair of dual hyperbolae in the corresponding Levi-Civita plane by the complex square mapping, see \cite[Thm. 4]{Takeuchi-Zhao Conformal}. 
In the case of parallel lines, these two lines are centrally symmetric, therefore the three points $O, P, \tilde{P}$ are collinear. {In the case of hyperbola, the points $P$ and $\tilde{P}$ lie on different branches of the hyperbola, and $P \tilde{P}$ is its major axis which necessarily contains $O$.} When the quadric $\mathcal{E}$ is degenerate, we may have a parabola as an intersection of $\mathcal{F}$ with the plane as well. A parabola is obtained as the image of a line by the complex square mapping \cite{Takeuchi-Zhao Conformal}, and the dual line is sent to the same parabola. In this case, we have $P=\tilde{P}$. }
\end{proof}

\begin{prop}\label{prop: quadrics_Hopf_image}  {The image of any $S^{1}$-invariant, non-singular, centered quadric in $\mathbb{H}$ under the Hopf mapping is either a plane, or a {centered} sphere, or a spheroid, or a sheet of a two-sheeted circular hyperboloid, or a paraboloid in $\mathbb{IH}$, with always a focus at the origin $O \in \mathbb{IH}$ in the latter three cases. }

\end{prop}

\begin{proof} 
We take an $S^{1}$-invariant, non-singular, centered quadric $\mathcal{E}$ in $\mathbb{H}$ and denote its image in $\mathbb{IH}$ by $\mathcal{F}$. {The quadrics $\mathcal{E}$ and $\mathcal{F}$ are bounded away from the origin $O$.} We intersect $\mathcal{F}$ with a plane trough $O \in \mathbb{IH}$. By the above theory of Levi-Civita planes, this plane is the image of an $S^{1}$-family of Levi-Civita planes on each of them the Hopf mapping restricts to the complex square mapping. The intersection of any of these Levi-Civita planes with the centered quadric in  $\mathbb{H}$ is a centered conic section. The image of this centered conic section is thus a branch of a conic section in the plane through the origin $O$ in $\mathbb{IH}$. In case that this branch is not a line {nor a circle},  then $O$ is a focus of it \cite{Takeuchi-Zhao Conformal}.

{We first assume that $\mathcal{E}$ is bounded in $\mathbb{H}$. Then its image $\mathcal{F}$ is also bounded in $\mathbb{IH}$. }If all points from $\mathcal{F}$ have the same distance to the origin, then $\mathcal{F}$ is a centered sphere in  $\mathbb{IH}$. Otherwise, there exist a point $P_{1}$ with least distance, and another distinct point $P_{2}$ with most distance from $O$. We consider the line passing through these two points and take a plane in $\mathbb{IH}$ containing both this line and the origin. By the above discussion on Levi-Civita planes, {the intersection of this plane with the image $\mathcal{F}$} is an ellipse focused at $O$. Consequently the indicated line passes through the origin, since for an ellipse this line is the major axis and passes through the foci. So the distance $|P_{1} P_{2}|$ is the major axis length of this ellipse. 

{We consider the family of planes passing through this line. If we take such a plane close to the plane we first took, then by continuity, the intersection of $\mathcal{F}$ on this plane is again an ellipse focused at $O$ and the points $P_1$ and $P_2$ lie on the ellipse as pericenter and as apocenter respectively. Thus the ellipses obtained as intersection of $\mathcal{F}$ on nearby planes from the family are related by a rotation around the line $P_{1} P_{2}$.   Consequently $\mathcal{E}$ is a spheroid with the line $P_{1} P_{2}$ as the symmetric axis.}

{This argument can be refined to the following local rigidity for (eccentric) ellipses, without assuming that $\mathcal{F}$ is bounded: Consider the line $P_{1} O$ and a plane through this line such that the intersection of $\mathcal{F}$ with it is an ellipse. Then the intersection of $\mathcal{F}$ with nearby planes through $P_{1} O$ are also ellipses, and these ellipses are obtained from each other by rotations along the axis $P_{1} O$. Indeed all these ellipse need to intersect $P_{1} O$ at the same point $P_{2}$, which necessarily is the apocenter for all of them. This implies this local rigidity for ellipses.}

Now we consider the case that $\mathcal{E}$ is not bounded, thus $\mathcal{F}$ is not bounded as well. {We take a point $P_{1} \in \mathcal{F}$ which has the least distance from $O$.}
{Since the centered quadric $\mathcal{E}$ is not given by a positive-definite quadric form, {its dual quadric} $\tilde{\mathcal{E}}$ is non-empty in $\mathbb{H}$. The image in $\mathbb{IH}$ of the dual $\tilde{\mathcal{E}}$ is $\tilde{\mathcal{F}}$. We take the point {$\tilde{P}_{1} \in \tilde{\mathcal{F}}$} which has the least distance from $O \in \mathbb{IH}$. From Lemma \ref{lem: least_dist_dual}, the three points $O, P_1, \tilde{P}_1$ {lie on} the same line. We consider the family of planes passing through this line. Since $\mathcal{F}$  is unbounded, there exist a plane in this family  which has unbounded intersection with $\mathcal{F}$. Thus the intersection of $\mathcal{F} \cup \tilde{\mathcal{F}}$ with this plane is either a pair of two centrally symmetric parallel lines, a pair of branches of a hyperbola with its focus at $O$, or a parabola with its focus at $O$. 

{In the case of a hyperbola, {note that we have the local rigidity just as in the elliptic case}: In a nearby plane from this family, the intersection of $\mathcal{F} \cup \tilde{\mathcal{F}}$ is again a hyperbola focused at $O$, with $P_1$ and $P_2$ as vertices at each branch. We conclude that $\mathcal{F}$ is a branch of a circular two-sheeted hyperboloid with a focus at $O$. }

{In the case of parallel lines, this local rigidity implies that $\mathcal{F}$ intersects nearby planes in lines with $P_{1}$ being the closed point from these lines to $O$. We conclude $\mathcal{F}$ is a plane perpendicular to the line $OP_{1}$.}}

{The only left case is when the intersection of $\mathcal{F}$ with a plane containing $OP_1$ is a parabola. This happens when the original quadric $\mathcal{E}$ is unbounded and degenerate. From the local rigidity of ellipses and hyperbolae, we conclude that if the intersection with a plane passing through $OP_{1}$ is a parabola, then the intersections of $\mathcal{F}$ with nearby planes passing through $OP_{1}$, we again obtain parabolae. These parabolae are focused at $O$ and have $P_1$ as the vertex. Thus, $\mathcal{F} $ intersects the nearby planes from this family in parabolae with the focus and the vertex in common. Thus in this case the image $\mathcal{F}$ is a paraboloid with a focus at $O$.}

\end{proof}

\begin{cor}
	{Any confocal combination of $\mathcal{S}^1$-invariant centered spheroids or two-sheeted circular hyperboloids in $\mathbb{H}$ is sent to a confocal combination of spheroid or a sheet of a two-sheeted circular hyperboloid.}
\end{cor}
\begin{proof}
{This follows from the fact that any confocal family of conic sections on a plane is sent to a confocal family of conic sections by the complex square mapping ( \cite[Thm. 4]{Takeuchi-Zhao Conformal} ) and the rotational symmetry of the images of centered quadrics with respect to the symmetry axis shown in Proposition \ref{prop: quadrics_Hopf_image}.}
\end{proof}

\begin{prop} \label{prop: 02}
	If a centered quadric in $\mathbb{H}$ is invariant under the $\mathcal{S}^1$-action on $\mathbb{H}$ given by 
	\begin{equation}
	\label{eq: S1_action_H}
	\theta \cdot z \mapsto \exp(i \theta)z, \quad z \in \mathbb{H},
	\end{equation}
	then it is a centered quadric given in the non-degenerate case by the normal form equation
	\begin{equation}
	\label{eq: centered_symmetric_quadrics_H nondeg}
	\frac{u_1^2}{A^2} \pm \frac{u_2^2}{B^2} + \frac{u_3^2}{A^2} \pm \frac{u_4^2}{B^2} =1, \quad A, B >0
	\end{equation}
	or in the degenerate case by the normal form equation
	\begin{equation}
	\label{eq: centered_symmetric_quadrics_H deg}
	\frac{u_1^2}{A^2} + \frac{u_3^2}{A^2}  =1, \quad A >0.
	\end{equation}
	The image under the Hopf mapping $z \mapsto Q= \bar{z}i \cdot z $ of such a centered quadric is a spheroid/a sheet of a circular hyperboloid in the non-degenerate case, including the sphere and plane as degeneracies, and a circular paraboloid in the degenerate case.

\end{prop}

\begin{proof}

	By Proposition \ref{prop: quadrics_Hopf_image}, the image of an $S^{1}$-invariant centered quadric is either a spheroid, or a sheet of a two-sheeted circular hyperboloid, or a paraboloid, all with a focus at the origin, or otherwise a centered sphere or a plane. We shall only discuss the case that this image is a spheroid. The other cases are similar. 
	
	The proof is computational.
	Up to normalization, a spheroid in $\mathbb{IH}$ focused at the origin is given by an equation of the form
	\[
	\frac{q_1^2 - \sqrt{C^2- D^2}}{C^2} + \frac{q_2^2}{D^2} + \frac{q_3^2}{D^2} - 1 = 0, \quad C>D>0.
	\]
	The mapping $z \mapsto Q= \bar{z}i \cdot z $ pulls this equation back to 
	\[
	G_1 \cdot G_2 = 0 
	\]
	where the factors are
	\[
	G_1:=  Cz_1^2 +  Cz_2^2 + Cz_3^2 +  Cz_4^2 -  2\sqrt{C^2- D^2} z_1 z_3  -  2\sqrt{C^2- D^2} z_2 z_4  - D^2 
	\]
	and
	\[
	G_2:=  Cz_1^2 +  Cz_2^2 +  Cz_3^2 +  Cz_4^2 + 2\sqrt{C^2- D^2} z_1 z_3  + 2\sqrt{C^2- D^2} z_2 z_4  + D^2. 
	\]
	It is readily seen that the equation $G_{2}=0$ does not admit any real solutions.
	
	In the rotated coordinates $(u_1,u_2,u_3,u_4)$ defined as 
	\[
	z_1 = \frac{u_1+u_2}{\sqrt{2}}, z_2 = \frac{u_3+u_4}{\sqrt{2}}, z_3 = \frac{u_1- u_2}{\sqrt{2}}, z_4 = \frac{u_3-u_4}{\sqrt{2}},
	\]
	we write
	\footnotesize
	\[
	G_1= (C-\sqrt{C^2 - D^2})u_1^2 + (C+\sqrt{C^2 - D^2})u_2^2 + (C-\sqrt{C^2 - D^2})u_3^2 + (C+\sqrt{C^2 - D^2})u_4^2 - D^2=0
	\]
	\normalsize
	and thus by a further normalization we get the desired form \eqref{eq: centered_symmetric_quadrics_H nondeg}.
	
	

	
	\end{proof}

In \cite{Panov}, it is noticed that conformal transformations between mechanical systems preserves billiard trajectories. A generalization of this observation to our current situation is the following:

\begin{prop} Let $\mathcal{R}$ be an $S^{1}$-invariant hypersurface in $\mathbb{H} \setminus O$ and $\tilde{\mathcal{R}} \subset \mathbb{IH}$ its image under the Hopf mapping. Let $v_{1}$ be an incoming vector at a point $z \in \mathcal{R}$ such that $(z, v_{1}) \in \Sigma^{1}$ with the outgoing vector $v_{2}$ after reflection. Then $(z, v_{2}) \in \Sigma^{1}$.  Assume that the Hopf mapping pushes $(v_{1}, v_{2})$ into $(\tilde{v}_{1}, \tilde{v}_{2})$. Then $\tilde{v}_{1}$ is reflected to $\tilde{v}_{2}$ by the reflection at $q=\bar{z} i z$ against $\tilde{\mathcal{R}}$.

In the opposite direction, if $\tilde{v}_{1}$ is reflected to $\tilde{v}_{2}$ by the reflection at $q$ against $\tilde{\mathcal{R}} \subset \mathbb{H} \setminus O$, then for any $z$ such that $q=\bar{z} i z$, there exists based vectors $v_{1}, v_{2}$ at $z$ such that $(z, v_{1}), (z, v_{2}) \in \Sigma^{1}$ which is pushed-forward into $(\tilde{v}_{1}, \tilde{v}_{2})$ by the Hopf mapping, such that $v_{1}$ is reflected to $v_{2}$ at $z$ against the pre-image $\mathcal{R}$ of $\tilde{\mathcal{R}}$.
 
\end{prop}

\begin{proof}
By assumption, we have
$$BL(z, v_{1})=0.$$
Consider the normal vector $N_{z}$ to $\mathcal{R}$ at $z$. Since $\mathcal{R}$ is $S^{1}$-invariant, we have that $N_{z}$ is orthogonal to the $S^{1}$-symmetric direction, which is given by $i z$. Consequently, we have 
$$BL(z, N_{z})={Re(\bar{z}iN_z)}=-\langle i z, N_{z} \rangle =0.$$
Since $BL$ is linear in its second variable, we conclude that 
$$BL(z, v_{2})=0$$
as well.

The second assertion follows as long as we show that the push-forward of $N_{z}$ is orthogonal to $\tilde{\mathcal{R}}$ at $q=\bar{z} i z$. The push-forward of a vector $v \in \Sigma$ is $\bl{2}\bar{z} i v$ . Thus
\begin{equation}\label{eq: equal angle}
\langle \bar{z} i v, \bar{z} i N_{v} \rangle ={|z|^2  \langle  v, N_{v} \rangle},
\end{equation}
meaning that the angle between $v$ and $N_{v}$ is preserved. Applying this for any vector $v \in \Sigma \cap T_{z} \mathcal{R}$ we conclude that $\bar{z} i N_{v}$ is orthogonal to $\tilde{\mathcal{R}}$ at $q$. 

For the opposite direction, if $\tilde{v}$ is a vector at $q \neq 0$ and $z \in \mathbb{H} \setminus O$ such that $q=\bar{z} i z$, then the vector $v$ such that $\bar{z} i v=\tilde{v}$ is a vector at $z$ which is pushed-forward to $\tilde{v}$. With this construction we get at each $z$ a pair of vectors $\{v_{1}, v_{2}\}$ from the pair of vectors $\{\tilde{v}_{1}, \tilde{v}_{2}\}$ at $q$. There follows directly that 
$$(z, v_{1}), (z, v_{2}) \in \Sigma^{1}.$$ 
Moreover it follows from the angle-preservation relationship \eqref{eq: equal angle} that if $\tilde{v}_{1}$ is reflected to $\tilde{v}_{2}$, then $v_{1}$ is reflected to $v_{2}$.

\end{proof}

As part of the proof, we have shown that if an orbit of the 4-dimensional Hooke problem satisfies the bilinear relation, then so is its reflection. Therefore we may say that a billiard orbit satisfies the bilinear relation. As only this type of orbits are related to the spatial Kepler problem, we propose the following definition.

\begin{defi} The subsystem of a 4-dimensional Hooke billiard consists only of orbits satisfying the bilinear relation is called the restricted 4-dimensional Hooke billiard.
\end{defi}

\begin{defi} A spatial Kepler billiard and a 4-dimensional Hooke billiard are called in correspondence, if the reflection wall of the Hooke problem in $\mathbb{H}$ is the pre-image of the reflection wall of the Kepler problem in $\mathbb{IH}$ by the Hopf map. 
\end{defi}

With {these definitions} we get the following theorem, which generalizes the planar Hooke-Kepler billiard correspondence as has been investigated in  \cite{Panov} and \cite{Takeuchi-Zhao Conformal}. 

\begin{theorem}\label{thm: 01} Any billiard orbit of the spatial Kepler billiard  is the image of an $S^{1}$-family of billiard orbits of the corresponding restricted 4-dimensional Hooke billiard. In the opposite direction, the image of any orbit of the restricted 4-dimensional Hooke billiard under the Hopf mapping is an orbit of the corresponding spatial Kepler billiard.
\end{theorem}

This theorem is not limited to the integrable case and thus may be useful to understand the dynamics of non-integrable 4-dimensional Hooke and 3-dimensional Kepler billiards.

For the integrable case, we know that a 4-dimensional Hooke billiard with a centered quadric reflection wall is integrable \cite{Jacobi Vorlesung}, \cite{Fedorov}. We directly obtain the following result, established in \cite{Takeuchi-Zhao Projective Space} via a completely different approach.


        
 \begin{cor} A spatial Kepler billiards with a branch of a focused quadric in $\R^{3}$ as reflection wall is integrable.
 \end{cor}

\section{The two-center problem and integrable billiards}\label{sec: 2CP}
 In this section, we consider the spatial two-center problem, which describes the motion of a particle in $\R^{3}$ to move under the gravitational attractions of two fixed centers. In the plane, this system is known to be integrable due to the works of Euler and Lagrange [REF], [REF]. The system is also integrable in $\R^{3}$. It is considered as a simplification of the planar or spatial circular restricted three-body problem with the Coriolis force and the centrifugal force ignored. 
 
In \cite{Birkhoff_1915}, Birkhoff designed a way to simultaneously desingularize the two double collisions of the particle with the two centers in the planar problem. This has been subsequently generalized to the spatial problem as first announced in Stiefel and Waldvogel \cite{Stiefel-Waldvogel}. In \cite{Waldvogel Diss} ,Waldvogel explained that the construction is analogous to the observation that on the Riemann sphere, Birkhoff's mapping is conjugate to the complex square mapping via a M\"obius transformation. The approach was then subsequently applied to the spatial problem. The use of quaternions were introduced in \cite{Waldvogel Right Way}. 

The goal of this section is to discuss this transformation in the spatial case with the language of the quaternions and symplectic geometry, with the hope to clarify the geometry of this transformation even further. Subsequently we apply this transformation to the problem of integrable billiards.


{We first recall Waldvogel's view of Birkhoff transformation of the planar two-center problem from \cite{Waldvogel Right Way}. See also \cite{Cieliebak-Frauenfelder-Zhao} for a discussion on the geometry of this transformation.}
 
 Consider the mappings
 \[
 \varphi_1: {\C \cup \{\infty\} \mapsto \C \cup \{\infty\}}, \quad z \mapsto \alpha = 1- \frac{2}{1-z},
 \]
 \[
 L.C. : {\C \cup \{\infty\} \mapsto \C \cup \{\infty\}}, \quad \alpha \mapsto q=\alpha^2,
 \]
 \[
 \varphi_2: {\C \cup \{\infty\} \mapsto \C \cup \{\infty\}}, \quad \mapsto x=1- \frac{2}{1-q}.
 \]
{ The mappings $\varphi_{1}$ and $\varphi_{2}$ are M\"obius transformations on the Riemann sphere $\C \cup \{\infty\}$. The mapping $L.C.$ is the complex square mapping, branched at $0, \infty$ on the Riemann sphere.}
 
 The composition of these mappings in the natural order gives rise to
 \[
 \varphi_2 \circ L.C. \circ {\varphi_1}: {\C \cup \{\infty\} \mapsto \C \cup \{\infty\}}, \quad z \mapsto x = \frac{z{+}z^{-1}}{2}. \]
This is Birkhoff's transformation, used to simultaneously regularize both double collisions with two Kepler centers placed at $-1, 1 \in \C$.
 
{This suggests the following construction for the spatial two-center problem.}

We define the \emph{base Birkhoff-Waldvogel mapping} as the composition 
\[
\begin{split}
\phi_2 \circ \hbox{Hopf} \circ \phi_1: &\mathbb{H} \cup \{\infty  \} \mapsto \mathbb{IH} \cup \{\infty  \},\\
&z \mapsto x=i-4\| z-i \|^4\left((z- \bar{z} -2i)\|z-i\|^2 +2(z-i)i(\bar{z}+i)\right)^{-1}
\end{split}
\]
where
\[
\begin{split}
\phi_1: & \mathbb{H} \cup \{\infty  \} \mapsto \mathbb{H} \cup \{\infty  \},\\ & z \mapsto \alpha =i - \frac{2}{z-i},
\end{split}
\]
\[
\begin{split}
\hbox{Hopf}: & \mathbb{H} \cup \{\infty  \} \mapsto \mathbb{IH} \cup \{\infty  \}\\ & \alpha \mapsto q=\bar{\alpha} i \alpha
\end{split}
\]
and 
\[
\begin{split}
\phi_2: &\mathbb{IH} \cup \{\infty  \} \to \mathbb{IH} \cup \{\infty  \},\\
&q \mapsto x = i - \frac{2}{q-i}.
\end{split}
\]

In coordinates, we have 
 \begin{equation}
 	\label{eq: BW_entrywise_formula}
 \begin{split}
 &x_1= \frac{1}{2}\left(z_1+ \frac{z_1(z_0^2+1)}{z_1^2 +z_2^2 +z_3^2}\right)\\
 &x_2 = \frac{1}{2}\left(z_2+ \frac{z_2(z_0^2-1) + 2 z_0 z_3}{z_1^2 +z_2^2 +z_3^2}\right)\\
 &x_3 = \frac{1}{2}\left(z_3+ \frac{z_3(z_0^2-1) - 2 z_0 z_2}{z_1^2 +z_2^2 +z_3^2}\right).
\end{split}
 \end{equation}

{By restriction and properly lifting the mappings to the cotangent bundles,} we get the \emph{unrestricted Birkhoff-Waldvogel mapping}
 \[
 \begin{split}
 \widetilde{B.W.}:=\Phi_2 \circ K.S. \circ \Phi_1: & {(\mathbb{H} \setminus \{i, -i\}) \times \mathbb{H} \to (\mathbb{H} \setminus \{i, -i\}) \times \mathbb{H}} ,
  (z, w) \mapsto (x,y)
 \end{split}
 \]
 where
 \[
 \begin{split}
 \Phi_1: & (\mathbb{H} \setminus \{i, {-i}\}) \times \mathbb{H} \to (\mathbb{H} \setminus \{{0,i}\}) \times \mathbb{H}, \\ &(z, w) \mapsto \left(\alpha= i - \frac{2}{z-i}, \beta= \frac{\overline{(z - i)}w \overline{(z- i)}}{2} \right),
 \end{split}
 \]
 \[
 \begin{split}
 K.S.: & (\mathbb{H} \setminus {\{0,i\}}) \times \mathbb{H} \to (\mathbb{H} \setminus {\{0,i\}}) \times \mathbb{H} \\&(\alpha,\beta) \mapsto \left(q=  \bar{\alpha}i \alpha,  p= \frac{ \bar{\alpha}i\beta }{2 |\alpha|^2}  \right),
 \end{split}
 \]
 \[
 \begin{split}
\Phi_2: &(\mathbb{H} \setminus \{{0}, i\}) \times \mathbb{H} \to (\mathbb{H} \setminus \{i, {-i}\}) \times \mathbb{H}  \\&(q, p) \mapsto \left(x= i - \frac{2}{q-i}, y= \frac{\overline{(q- i)}p \overline{(q- i)}}{2}\right).
 \end{split}
 \]
 Explicitly, the unrestricted Birkhoff-Waldvogel mapping $\widetilde{B.W.}$ is given by $(z, w) \mapsto (x, y)$ with
 \scriptsize
 \[
 {\begin{split}
 	&x=
 	i-|z-i|^{2}(2|z-i|^{-2}(zi\bar{z}+ \bar{z} -z +i)-\bar{z} +z -2i)^{-1}\\
 	&y = \frac{1}{|i - 2(z-i)^{-1}|^2}((z-i)^{-1}(\bar{z}+i)^{-1}i
 	 (\bar{z}+i)-i(\bar{z}+i)^{-1}+(z-i)^{-1}i+2(z-i)^{-1}(\bar{z}+i)^{-1})\\
 	&\times w(1-(\bar{z}+i)(z-i)^{-1}+2i(z-i)^{-1}).
 \end{split}}x
 \]
 \normalsize
 
{The mapping $\Phi_{1}, \Phi_{2}$ are constructed in a way that the transformations on the positions are natural generalizations of $\phi_{1}, \phi_{2}$, while the transformations on momenta are obtained as contragradients. The mapping $K.S.$ is the usual Kustaanheimo-Stiefel transformation.} 

In $(\mathbb{H}\setminus \{ i \}) \times \mathbb{H}$ we define the subsets
 \[
 \hat{\Lambda}:= \{(z,w) \in (\mathbb{H} \setminus (\R \cup \{i\} {\cup \{-i\}})) \times \mathbb{H} \mid Re((\bar{z}-i)w(\bar{z}+i)) = 0  \}
 \]
 and
 \[
 \hat{\Sigma} := \{ (\alpha,\beta) \in (\mathbb{H}\setminus (\{e^{i \theta}\} {\cup \{0\}})) \times \mathbb{H} \mid BL(\alpha, \beta) = Re(\bar{\alpha}i\beta)  = 0  \}.
 \]
 Then we have the following.

 \begin{lemma}\label{lemma: 16}
    The image of the mapping $\Phi_1$ with the restricted domain $\hat{\Lambda}$ is $\hat{\Sigma}$.
 	 Additionally, the image of the mapping $\Phi_2$ restricted to $(\mathbb{IH}\setminus \{i, {0}\})\times \mathbb{IH}$ is $(\mathbb{IH}\setminus \{i,{-i}\}) \times \mathbb{IH}$.
 \end{lemma}

To show this we first show
\begin{lemma} \label{lemma: 17} ${\phi_1^{-1}(\{ e^{i \theta} \}) = \R}$. 
\end{lemma} 

Since $K.S.(\{ e^{i \theta} \})=i$ and $\phi_{2}(i)=\infty$, this shows in particular that $\R \subset \mathbb{H}$ represents physical infinity of the physical space $\mathbb{IH}$. 
  
   \begin{proof}
  The pre-image of $\alpha =  e^{i \theta}$ is
  \[
  \begin{split}
  z= i- \frac{2}{\alpha-i}  &= i-\frac{2}{\cos \theta - i (1- \sin \theta)}\\ & = i - \frac{2(\cos \theta + i (1- \sin \theta))}{(\cos \theta - i (1- \sin \theta))(\cos \theta + i (1- \sin \theta))}\\
  & = i-\frac{\cos \theta + i (1- \sin \theta) }{1- \sin \theta} \\
  & = \frac{\cos \theta}{\sin \theta-1}.
  \end{split}
  \]
We thus have
  \[
  \left\{    z=\frac{\cos \theta}{\sin \theta-1}\mid \theta \in \R/2\pi \Z \right\} =\R. 
  \]
  \end{proof}
  
  \begin{proof} (of Lemma \ref{lemma: 17} )
  The image $\Phi_1(\hat{\Lambda})$ is contained in $\hat{\Sigma}$, since
  \begin{equation}
  \label{eq: BL_equiv_condition}
  \begin{split}
  	Re(\bar{\alpha} i \beta) =  0 &\Leftrightarrow  Re \left(-i - \frac{2}{\bar{z}+i}\right)i \left(\frac{(\bar{z}+i)w(\bar{z}+i)}{2}\right) = 0 \\
   &\Leftrightarrow  Re((1-2(\bar{z}+i)^{-1}i)(\bar{z}+i)w(\bar{z}+i)){=0}\\
   &\Leftrightarrow  Re((\bar{z}+i)^{-1}(\bar{z} +i -2i)(\bar{z}+i)w(\bar{z}+i)) {=0} \\
   &\Leftrightarrow  Re((\bar{z} - i)w(\bar{z}+i))=0. 
  \end{split}
  \end{equation}
 On the other hand, for any $(\alpha,\beta) \in \hat{\Sigma}$,  {its preimage} $(z,w) \in \hat{\Lambda}$ by $\Phi_1$ {is
     given by the formulas}
  \[
 z= \frac{2}{i - \alpha } + i
 \]
 and 
 \[
 w =  2 (\bar{z}+i)^{-1} \beta (\bar{z} + i )^{-1} = 2 \alpha \beta \alpha.
 \]
Thus, the first part of the lemma follows.

  For $(q,p) \in (\mathbb{IH}\setminus \{i\}) \times \mathbb{IH}$, the conjugation 
   of its image $(x,y)$ by $\Phi_2$ is obtained as
  \[
  \bar{x} = -i - \frac{2}{-q+i}= -(i - \frac{2}{q-i}) = - x
  \]
  and 
  \[
  \bar{y}= - \frac{(q-i)p(q-i)}{2} = - \frac{(-q+i)p(-q+i)}{2} = - \frac{(\overline{q-i})p(\overline{q-i})}{2} = -y,
  \]
  thus $(x, y) \in (\mathbb{IH}\setminus \{i\}) \times \mathbb{IH}$. 
  
  On the other hand, for any $(x, y) \in (\mathbb{IH}\setminus \{i\}) \times \mathbb{IH}$, the conjugate 
  of its preimage $(p,q)$ is obtained as
  \[
  \bar{q} =\frac{2}{-i + x} + i = - \left( \frac{2}{i-x} +i  \right) = -q
  \]
  and 
  \[
  \bar{p}=  -2 xyx = -p,
  \]
  thus the preimage $(p,q)$ belongs again to $(\mathbb{IH} \setminus\{i\})
   \times \mathbb{IH}$.
  
\end{proof}

\begin{lemma}
	The image of the $K.S$ mapping restricted to $\hat{\Sigma} $ is $(\mathbb{IH}\setminus \{i, {0}\}) \times \mathbb{IH}$.
\end{lemma} 
\begin{proof}
	The image $K.S.(\hat{\Sigma})$ is included in $\mathbb{IH} \times \mathbb{IH}$ since 
	$\bar{\alpha} i \alpha \in \mathbb{IH}$ for any $\alpha \in \mathbb{H}$ and
	\[
	BL(\alpha, \beta )=0 \Leftrightarrow Re(\bar{\alpha} i \beta)= Re(p)=0.
	\]
	On the other hand, for any $(q,p) \in (\mathbb{IH}\setminus \{i\}) \times \mathbb{IH}$, we can take $(\alpha, \beta) \in \hat{\Sigma}$ such that $K.S.(\alpha, \beta) = (q,p)$. Indeed,for any $(q,p) \in \mathbb{IH} \times \mathbb{IH}$, there exist an $S^1$-family $\{(e^{i \theta_1}\alpha, e^{i \theta_1}\beta  ) \}$ satisfying $BL(\alpha, \beta) =0$.
\end{proof}
 {From these lemmas, we obtain the following proposition:}
 \begin{prop}
 {The \emph{restricted Birkhoff-Waldvogel mapping}
 \[
 B.W.: \hat{\Lambda} \to (\mathbb{IH}\setminus \{i,-i\}) \times \mathbb{IH} \quad (z,w) \mapsto (x, y),
 \]
 where 
 \[
 \begin{split}
 	&x=i-|z-i|^{2}(2|z-i|^{-2}(zi\bar{z}+ \bar{z} -z +i)-\bar{z} +z -2i)^{-1}\\
 	&y = \frac{1}{|i - 2(z-i)^{-1}|^2}((z-i)^{-1}(\bar{z}+i)^{-1}i (\bar{z}+i)-i(\bar{z}+i)^{-1}+(z-i)^{-1}i+2(z-i)^{-1}(\bar{z}+i)^{-1})\\
 	&\times w(1-(\bar{z}+i)(z-i)^{-1}+2i(z-i)^{-1})
 \end{split}
 \] 
 is surjective.}
\end{prop}
 The following proposition describes the symplectic property of the restricted Birkhoff-Waldvogel mapping:
 
 \begin{prop}
	{$B.W.^* (Re(d \bar{y} \wedge dx)) = Re(d \bar{w}\wedge dz) |_{{\hat{\Lambda}}}$.}
 \end{prop}
\begin{proof}
	We compute the 1-form:
	\[
	\begin{split}
	 \Phi_1^{*}(Re(\bar{\beta} d \alpha)) &= Re \left( \frac{(z-i)\bar{w}(z-i)}{2} \cdot (-2d(z-i)^{-1})   \right) \\
	&= Re \left( \frac{(z-i)\bar{w}(z-i)}{2} \cdot2(z-i)^{-1}(d(z-i)) (z-i)^{-1}    \right)\\
	&=Re(\bar{w} dz)
	\end{split}
	\]
	Similarly, we get
	\[
	\Phi_2^{*}(Re(\bar{y} dx)) = Re(\bar{p} dq).
	\]
	We now recall the fact
	\[
	K.S.|_{\Sigma}^* (Re(\bar{p} d q))= Re(\bar{\beta} d \alpha).
	\]
	Since $\hat{\Sigma} \subset \Sigma$, we have 
	\[
	K.S.|_{\hat{\Sigma}}^* (Re(\bar{p} d q)) = Re(\bar{\beta} d \alpha).
	\]
	By combining these facts, we obtain
	\[
	B.W.^{*}(Re(\bar{y}dx)) = Re(\bar{w} dz).
	\]
\end{proof}
 
{We now apply this mapping to the two center problem in $\R^{3} \cong \mathbb{IH}$, with the two centers at $\pm i \in  \mathbb{IH}$.} We start with the shifted-Hamiltonian of the two-center problem
 \[
 H - f = \frac{|y|^2}{2} + \frac{m_1}{|x-i|} + \frac{m_2}{|x+i|} -f 
 \]
 and consider its 0-energy hypersurface. By multiplying the above equation by $|x-i||x+i|$, we obtain 
 \[
 |x-i||x+i|(H-f)= \frac{|y|^2|x-i||x+i|}{2} + m_1|x+i|+ m_2|x-i| -f|x-i||x+i|.
 \]
With the following {identities}
 \[
 \begin{split}
 &B.W.^{*} (|x-i|) = \frac{|z-i|^2}{|\bar{z} -z|}\\
 &B.W.^{*} (|x+i|)= \frac{|z+i|^2}{|\bar{z} -z|}\\
 &B.W.^{*} (|y|^2) = \frac{|\bar{z}-z|^4 |w|^2}{4|z-i|^2|z+i|^2}
 \end{split}
 \]
 we obtain
\begin{equation}\label{eq: def tilde K}
  \tilde{K}=\frac{|w|^2|\bar{z}-z|^2}{8} + m_1 \frac{|z+ i|^2}{|\bar{z} -z|} + m_2\frac{|z- i|^2}{|\bar{z} -z|} -f \frac{|z-i|^2|z+i|^2}{|\bar{z}-z|^2}=0,
 \end{equation}
 which can be put in the standard form of a natural mechanical system in the plane by a further multiplication of $|\bar{z}-z|^{-2}$: In this way we get
\begin{equation}\label{eq: def K}
K:= \frac{|w|^2}{8} + m_1 \frac{|z+ i|^2}{|\bar{z} -z|^3} + m_2\frac{|z- i|^2}{|\bar{z} -z|^3} -f \frac{|z-i|^2|z+i|^2}{|\bar{z}-z|^4}=0.
 \end{equation}
 
Note that the Hamiltonian \eqref{eq: def K} is regular at the physical double collisions $\{z=\pm i\}$. The physical collisions are therefore regularized. Its singular set $\{z \in \R\}$ corresponds to $\infty$ of the original system, and is not contained in any finite energy level {(Lem. \ref{lemma: 17})}.

 \begin{prop} 
 	{Consider a plane in $\mathbb{IH}$ containing the $i$-axis given by the equation
 	\begin{equation}
 	\label{eq: plane_IH}
 	k_2 x_2 + k_3 x_3 = 0
 	\end{equation}
 	with $(k_1, k_2) \in \R^{2} \setminus O$. The pre-image of this plane by the B.W. mapping is the family of 2-dimensional spheres and planes given by
	 	\begin{align}
 		\left\{
 		\begin{aligned}\label{eq: family of spheres and planes}
 			& (\sin \theta z_0 - \cos \theta )^2 + (z_1^2 +z_2^2 +z_3^2) \sin^{2} \theta = 1 \\
 			& k_2 (z_2 \cos \theta  + z_3 \sin \theta )+ k_3(z_3 \cos \theta  - z_2 \sin \theta) = 0. \\
 		\end{aligned}
 		\right.
 	\end{align}

        For each $\theta \not\equiv 0, \pi\, (mod \,\, 2 \pi)$, Equation \eqref{eq: family of spheres and planes} describes a 2-dimensional sphere as the intersection of a three-dimensional sphere with a hyperplane in $\mathbb{H}$. We call them \emph{Birkhoff spheres.} We denote them by $S_{\theta, \kappa}$ respectively.
        
        For $\theta \equiv 0, \pi\, (mod \,\, 2 \pi)$, Equation \eqref{eq: family of spheres and planes} describes the plane
 	\begin{align}
 		\left\{
 		\begin{aligned}
 			& z_0 = 0 \\
 			& k_2  z_2 + k_3 z_3  = 0,
 		\end{aligned}
 		\right.
 	\end{align}
	which we call a \emph{Birkhoff plane} and we denote it by $\pi_{\kappa}$.
	In both cases, the angle $\kappa$ is the unique angle which satisfies 
	$$\cos \kappa=z_{2}, \sin \kappa=z_{3}.$$
	Moreover, the mapping $B. W. $ is restricted to the Birkhoff mapping on the cotangent bundle of a Birkhoff plane.
 } 	
 
   \end{prop}
 	\begin{proof}
 	The pre-image of the plane \eqref{eq: plane_IH} by the mapping $\phi_2$ is a plane given by 
 	\[
 	k_1 q_2 + k_2 q_3 = 0
 	\]
 	in $\mathbb{IH}$.
 	{The pre-image of this plane by the Hopf map is the family of Levi-Civita planes given by 
 	\begin{align}
 		\left\{
 		\begin{aligned}
 			&  \alpha_0 \cos \theta + \alpha_1  \sin \theta =0\\
 			& k_2 (\alpha_2 \cos \theta  + \alpha_3 \sin \theta )+ k_3(\alpha_3 \cos \theta  - \alpha_2 \sin \theta  ) = 0 \\
 		\end{aligned}
 		\right.
 	\end{align}
 	in which $\theta \in \R/2\pi \Z$ is an angle parametrizing the $S^{1}$-symmetry of the Hopf mapping. The pre-image of this family of Levi-Civita planes by $\phi_{1}$ is 
 	\begin{align}
 		\left\{
 		\begin{aligned}
 			& \cos \theta (-z_0) + \sin \theta \left(z_1 - 1 + (z_0^2 + (z_1 -1)^2 + z_2^2 + z_3^2)^{2}/2 \right) =0,\\
 			& k_2 (\cos \theta z_2 + \sin \theta z_3)+ k_3(\cos \theta z_3 - \sin \theta  z_2) = 0, \\
 		\end{aligned}
 		\right.
 	\end{align}
 which is equivalent to Eq. \ref{eq: family of spheres and planes}.} 
	For the last assertion, the restriction of the $B.W.$ mapping to a Birkhoff plane is the composition of planar mappings each of them can be identified with $\varphi_{2}, L.C., \varphi_{1}$ respectively. 
	{Indeed, the restriction of $\phi_1$ to the $ij$-plane is obtained as 
	\[
	\phi_1(z_1i + z_2j) = (1 - 2((z_1-1)^2+z_2)^{-1}(1-z_1))i + 2((z_1-1)^2+z_2)^{-1}z_2 j
	\]
	which is equivalent to the Möbius transformation on $\mathbb{C} \cup \{ \infty \}$ given by
	\[
	\varphi_1(z_0 + z_1 i) = 1-2((z_0 - 1)^2 + z_1^2)^{-1}(1-z_0) - 2(z_0 - 1)^2 + z_1^2)^{-1}z_1 i
	\]
	up to some basis changes. One can generalize this identification to any planes in $\mathbb{IH}$ containing the $i$-axis by rotating the plane with respect to the $i$-axis. Analogously, we can identify the restriction of $\phi_2$ to a plane in $\mathbb{IH}$ containing the $i$-axis. Finally, we recall the argument from Proposition \ref{prop: LC-Hopf} and use the equivalence between the restriction of the Hopf mapping to the $ij-$plane in $\mathbb{IH}$ and the complex square mapping.
	}
	The conclusion follows.
	
 	\end{proof}
	
	It is desirable to relate $S_{\theta, \kappa}$ and $\pi_{\kappa}$, as they are related by the symmetry of the Birkhoff-Waldvogel mapping. We also would like to introduce natural coordinates to analyze the transformed system. For this purpose, we have the following lemma:
	
\begin{lemma} Let $z \in \pi_{k_{2}, k_{3}}$ be expressed as 
$$z= (r \cos \psi) \, i+(r \sin \psi \cos \kappa) \, j+(r \sin \psi \sin \kappa) \, k,$$ 
and $z_{\theta} \in S_{\theta, \kappa}$ be related to $z$ by the action of the $S^{1}$-symmetry of the Birkhoff-Waldvogel mapping by shifting the corresponding angle by $\theta$. Then we have 
\footnotesize
\begin{equation}\label{eq: z_theta}
z_{\theta}=\dfrac{(1-r^2) \sin \theta}{ (r^2+1)-(r^2-1) \cos \theta}+\dfrac{2 r \cos \psi}{ (r^2+1)-(r^2-1) \cos \theta} i+\dfrac{2 r \sin \psi \cos(\theta+\kappa)}{ (r^2+1)-(r^2-1) \cos \theta} j+\dfrac{2 r \sin \psi \sin(\theta+\kappa)}{ (r^2+1)-(r^2-1) \cos \theta} k.
\end{equation}
\normalsize
\end{lemma}	

\begin{proof} The mapping $\pi_{\kappa} \to S_{\theta, \kappa}, z \mapsto z_{\theta}$ is computed as $z_{\theta}=\phi^{-1}_{1}(e^{i\theta}\phi_{1} (z))$. This leads to the formula above.
\end{proof}

We may thus use $(r, \psi, \kappa, \theta)$ as coordinates for points in $\mathbb{H} \setminus  O$  with the help of Eq. \eqref{eq: z_theta}. The mapping $(r, \psi, \kappa, \theta) \mapsto z:=z_{\theta}$ is seen to be 2-to-1, as both $(r, \psi, \kappa, \theta)$ and $(r, \psi, \kappa, \theta+\pi)$ is sent to the same point $z \in \mathbb{H}$.

We compute $\tilde{K}$ in Eq. \eqref{eq: def tilde K} with these coordinates. We denote by $(P_{r}, P_{\psi}, P_{\kappa}, P_{\theta})$ the corresponding conjugate momenta. We set $P_{\theta}=0$, which is equivalent to the condition $Re((\bar{z}+i)w (\bar{z}-i))=0$. 
{This follows from Eq.\eqref{eq: BL_equiv_condition}.}
We then obtain after this restriction the formula
\begin{equation} \label{eq: formula for tilde K in spherical coordinates}
\tilde{K}=\dfrac{r^{2} P_{r}^{2}}{2}+\dfrac{ P_{\psi}^{2}}{2}+\dfrac{P_{\kappa}^{2}}{ 2 \sin^{2} \psi} -\dfrac{2 r^2 P_{\kappa}^{2}}{(r^2-1)^2} +4 f \cos^2 \psi+(m_1-m_2) \cos \psi +\dfrac{(f r^2-( m_1+ m_2) r/2+f) (r^2+1)}{r^2}.
\end{equation}
Which can be considered as the reduced system with respect to the $S^{1}$-symmetry in the direction of $\theta$.

We have 
\begin{equation} \label{eq: separation in two parts for K}
\tilde{K}=\tilde{K}_{1}+\tilde{K}_{2},
\end{equation}
with

\begin{equation} \label{eq: expressions of each parts for K}
\begin{aligned}
&\tilde{K}_{1}(r, P_{r}, P_{\kappa})=\dfrac{r^{2} P_{r}^{2}}{2} -\dfrac{2 r^2 P_{\kappa}^{2}}{(r^2-1)^2} +\dfrac{(f r^2-( m_1+ m_2) r/2+f) (r^2+1)}{r^2} ;\\
&\tilde{K}_{2}(\psi, P_{\psi}, P_{\kappa})= \dfrac{ P_{\psi}^{2}}{2}+\dfrac{P_{\kappa}^{2}}{ 2 \sin^{2} \psi}+4 f \cos^2 \psi+(m_1-m_2) \cos \psi .
\end{aligned}
\end{equation}

The angle $\kappa$ does not appear in this formula, reflecting the rotational invariance of the system around the axis of centers in $\mathbb{IH}$. We may thus fix $P_{\kappa}=C$. The further reduced Hamiltonian is 
\begin{equation} \label{eq: separation in two parts}
\tilde{K}_{red}=\tilde{K}_{red, 1}+\tilde{K}_{red, 2},
\end{equation}
with

\begin{equation} \label{eq: expressions of each parts}
\begin{aligned}
&\tilde{K}_{red, 1}(r, P_{r})=\dfrac{r^{2} P_{r}^{2}}{2} -\dfrac{2 r^2 C^{2}}{(r^2-1)^2} +\dfrac{(f r^2-( m_1+ m_2) r/2+f) (r^2+1)}{r^2} ;\\
&\tilde{K}_{red, 2}(\psi, P_{\psi})= \dfrac{ P_{\psi}^{2}}{2}+\dfrac{C^{2}}{ 2 \sin^{2} \psi}+4 f \cos^2 \psi+(m_1-m_2) \cos \psi .
\end{aligned}
\end{equation}

Both $\tilde{K}_{red, 1}(r, P_{r})$, $\tilde{K}_{red, 2}(\psi, P_{\psi})$ are 1 degree of freedom systems. The theory of \cite{Takeuchi-Zhao Conformal} applies. Any finite combination of concentric circles centered at the origin and lines through the origin in the $(r, \psi)$-plane are integrable reflection walls. 

It follows from Eq. \eqref{eq: z_theta} that each fibre of the $B. W.$-mapping intersects the subspace $\mathbb{IH}$ in two points when $r \neq 1$, and lie completely in this subspace when $r =1$.  In this latter case, only the combination of the angles $\theta+\kappa$ appears in the formula, meaning that in this case the $\kappa$-orbit is the same as the $\theta$-orbit. This is reflected in the formula \eqref{eq: formula for tilde K in spherical coordinates} for $\tilde{K}$, which is singular at $\{r=1\}$ if $C \neq 0$. Indeed it is not hard to check that this set corresponds to the $i$-axis in the physical space $\mathbb{IH}$. {This follows from Eq.\eqref{eq: BW_entrywise_formula}.} Otherwise, it is also singular at $\psi=0, \pi \,(mod\, 2 \pi)$, corresponding again to the $i-$axis. With this in mind, we consider the restriction of the system to the set 
$$\tilde{D}=\{z=z_{1} i + z_{2} j + z_{3} k \in \mathbb{IH} \setminus: |z| \neq 0, 1, (z_{2}, z_{3})\neq (0,0)\}$$
with (orthogonal) spherical coordinates $(r, \psi, \kappa)$, given by Eqs. \eqref{eq: separation in two parts for K}, \eqref{eq: expressions of each parts for K}.



\begin{prop} A mechanical billiard system in $\mathbb{IH}$, defined by the restriction of $\tilde{K}$ and any finite combination of concentric spheres and any cones symmetric around the $i$-axis {with the vertex at the origin} is integrable.
\end{prop}

\begin{proof} 
	{The spherical coordinates $(r, \psi, \kappa)$ are orthogonal. At a point of reflection we decompose the velocity as $v_{r} \,\vec{e}_{r} + v_{\psi}\, \vec{e}_{\psi}+v_{\kappa}\, \vec{e}_{\kappa}.$
	We consider a sphere centered at the origin $O$ or a cone symmetric around the $i$-axis with a vertex at $O$ as reflection wall. Due to the symmetry of the wall with respect to the $i$-axis, the $\kappa$-component of the velocity, thus the $P_\kappa$ does not change under reflections. Also, the intersection of the wall and a plane containing the $i$-axis is a circle centered at the origin or a line passing through the origin, thus the argument in \cite[Lemma 3]{Takeuchi-Zhao Conformal} applies and we see that both $P_r^2$ and $P_{\psi}^2$ are conserved. Therefore, we conclude that $\tilde{K}_1$ and $\tilde{K}_2$ are conserved.}
	
\end{proof}



By an r-confocal quadric in $\mathbb{IH}$ we mean a spheroid or a circular hyperboloid of two sheets there-in with foci at the two Kepler centers $\pm i$. Restricting the system in $\mathcal{D}$ to the Birkhoff planes and making use of \cite{Takeuchi-Zhao Conformal}, we obtain that the above mentioned system is equivalent to the two-center billiards in $\R^{3}$ with any combinations of r-confocal quadrics as reflection walls.

\begin{theorem} \label{thm: 22} The above-mentioned mechanical billiard system is equivalent to the two-center billiards in $\R^{3}$ with any combinations of r-confocal quadrics as reflection walls.
\end{theorem}

Thus this provides an alternative way to show the integrability of these two-center billiards, and generalizes \cite{Takeuchi-Zhao Conformal} to dimension 3.

\begin{rem}
	The Lagrange problem in $\R^3$ is given by the Hamiltonian
	\[
	H  = \frac{|y|^2}{2} + m_0 |x|^2 + \frac{m_1}{|x-i|} + \frac{m_2}{|x+i|},
	\]
	with $m_0, m_1, m_2 \in \R$ as parameters. 
	
	The same procedure shows that this system is separable after reduction in the same coordinates as above. Consequently, Thm. \ref{thm: 22} also holds with the Lagrange problem as the underlying system, as well as for other similar systems separable after reduction in these coordinates.

\end{rem}

\textbf{Acknowledgement} A.T. and L.Z. are supported by DFG ZH 605/1-1, ZH
605/1-2.

	\hspace{-1cm}
	\begin{tabular}{@{}l@{}}%
		Airi Takeuchi\\
		\textsc{University of Augsburg, Augsburg, Germany.}\\
		\textit{E-mail address}: \texttt{airi1.takeuchi@uni-a.de}
	\end{tabular}
	\vspace{10pt}
	
	\hspace{-1cm}
	\begin{tabular}{@{}l@{}}%
		Lei Zhao\\
		\textsc{University of Augsburg, Augsburg, Germany.}\\
		\textit{E-mail address}: \texttt{lei.zhao@math.uni-augsburg.de}
	\end{tabular}

\end{document}